\documentclass[leqno]{amsart}
\usepackage{amsmath}
\usepackage{amssymb}
\usepackage{amsthm}
\usepackage{enumerate}
\usepackage[mathscr]{eucal}
\usepackage{mathrsfs}
\usepackage{xcolor}
\theoremstyle{plain}
\usepackage{tikz}
\newtheorem{theorem}{Theorem}[section]

\newtheorem{prop}[theorem]{Proposition}
\theoremstyle{definition}

\newtheorem{remark}[theorem]{Remark}

\newtheorem{example}[theorem]{Example}

\newtheorem{cor}[theorem]{Corollary}
\theoremstyle{remark}

\begin{document}
	
	\title [On joint numerical radius and joint numerical index of a Banach space  ]{On joint numerical radius of operators and joint numerical index of a Banach space}

	\author[Arpita Mal]{Arpita Mal}
	
	\address[]{Department of Mathematics\\ Indian Institute of Science\\ Bangalore 560012\\ India.}
	\email{arpitamalju@gmail.com}

	
	\thanks{The author would like to thank SERB, Govt. of India for the financial support in the form of National Post Doctoral Fellowship under the mentorship of Prof. Apoorva Khare.}

	
	
	

	\date{}
	\maketitle
	\begin{abstract}
	Generalizing the notion of numerical range and numerical radius of an operator on a Banach space, we introduce the notion of joint numerical range and joint numerical radius of tuple of operators on a Banach space. We study the convexity of the joint numerical range. We show that the joint numerical radius defines a norm if and only if the numerical radius defines a norm on the corresponding space. Then we prove that on a finite-dimensional Banach space, the joint numerical radius can be retrieved from the extreme points. Furthermore, we introduce a notion of joint numerical index of a Banach space. We explore the same for direct sum of Banach spaces. Applying these results, finally we compute the joint numerical index of some classical Banach spaces.
	\end{abstract}

	\section{Introduction}
    The purpose of this article is to generalize the notion of numerical range and numerical radius of an operator on a Banach space to $k$-tuple of operators and explore the joint behavior of operators through these concepts. Moreover, we  generalize the notion of numerical index of a Banach space and study the related areas. To proceed further, we introduce the relevant notations and terminologies.\\
	Suppose $H$ is a Hilbert space and $X,Y$ are  Banach spaces. Unless otherwise mentioned, we always assume that the scalar field $F$ is either real or complex. Let $S_X=\{x\in X:\|x\|=1\},$ $B_X=\{x\in X:\|x\|\leq1\}$ and $E_X$ be the set of all extreme points of $B_X.$  $X^*$ denotes the dual space of $X.$ Suppose $L(X,Y)$ denotes the space of all bounded linear operators from $X$ to $Y,$ endowed with the usual operator norm. If $X=Y,$ then we simply write $L(X)$ instead of $L(X,X).$ The symbols $I$ and $O$ denote respectively the identity operator and the zero operator on the corresponding space. For $T\in L(H),$ the numerical range $W(T)$ and the numerical radius $w(T)$ associated with the operator $T$ are defined as follows:
	\[W(T)=\{\langle Tx,x\rangle:x\in S_H\}, ~w(T)=\sup\{|\langle Tx,x\rangle|:x\in S_H\}.\]
	The heart of the theory of  numerical range lies in the Toeplitz-Hausdorff Theorem, which says that $W(T)$ is always convex for $T\in L(H).$ If $H$ is a complex Hilbert space, then $w(\cdot)$ defines a norm on $L(H).$ On the other hand, if $H$ is a real Hilbert space, then $w(\cdot)$ defines a semi norm on $L(H).$  For other nice properties of $W(T)$ and $w(T),$ we refer the readers to \cite{GR,H}.
	Lumer \cite{L} and Bauer \cite{B} generalized the notion of the numerical range of an operator on a Banach space. For $T\in L(X),$ the numerical range $W(T)$ and the numerical radius $w(T)$ are defined as
	\[W(T)=\{x^*(Tx):(x,x^*)\in \Pi_X\},~w(T)=\sup\{|x^*(Tx)|:(x,x^*)\in \Pi_X\},\]
	where $\Pi_X=\{(x,x^*)\in S_X\times S_{X^*}:x^*(x)=1\}.$ In contrast to the Hilbert space, if $T\in L(X),$ $W(T)$ may not be always convex. Note that, $w(\cdot)$ is always a semi norm on $L(X).$ In particular, if $X$ is a complex Banach space, then it defines a norm. However, there are some real Banach spaces $X$ where $w(\cdot)$ is a norm on $L(X).$ Classical references for the theory of numerical range on a Banach space and related areas are \cite{BD,BD2}. There is a constant on a Banach space, known as the numerical index of the space, which relates the behavior of the numerical radius with the usual norm of an operator.  The numerical index $n(X)$ of a space $X$ is defined as follows:
	\[n(X)=\inf\{w(T):T\in L(X),\|T\|=1\}.\]
	Let us recall a few facts on $n(X),$ which we will use the sequel without mentioning further. If $X$ is a complex Banach space, then $\frac{1}{e}\leq n(X)\leq 1$ and if $X$ is a real Banach space, then $0\leq n(X)\leq 1.$ If $H$ is a complex Hilbert space, then $n(H)=\frac{1}{2}$ and for a real Hilbert space $H,$ $n(H)=0.$ Moreover, $n(\ell_1)=n(\ell_1^m)=n(\ell_\infty)=n(\ell_\infty^m)=1,$ where $m\in \mathbb{N}.$ All these results can be found in \cite{BD,KMP}. Some recent references for the study of the numerical index are \cite{MMQ,MQ,SPBB}.
	
	For a $k$-tuple $\mathcal{T}=(T_1,\ldots,T_k)$ of operators on $H,$ the joint numerical range $W(\mathcal{T})$ of $\mathcal{T}$ is a subset of $F^k,$ defined as follows:
	\[W(\mathcal{T})=\{(\langle T_1x,x\rangle,\ldots,\langle T_1x,x\rangle ):x\in S_H\}.\]
	In general, the joint numerical range is not convex (see \cite[Ex. 3.1]{MT}). To get an overview of other interesting properties of the joint numerical range of Hilbert space operators,  the readers may follow the survey article \cite{MT} and the references therein.  In \cite{P} Popescu defined the joint numerical radius $w(\mathcal{T})$ of $\mathcal{T},$ also known as the euclidean operator radius $w_e(\mathcal{T})$, by
	\[w(\mathcal{T})=\sup\Bigg\{\Big(\sum_{i=1}^k|\langle T_ix,x\rangle|^2\Big)^{1/2}:x\in S_H\Bigg\}.\]
	In \cite{MSS}, Moslehian et al. extended the notion of euclidean operator radius in the following way:
		\[w_p(\mathcal{T})=\sup\Bigg\{\Big(\sum_{i=1}^k|\langle T_ix,x\rangle|^p\Big)^{1/p}:x\in S_H\Bigg\}, ~\text{where~}p\geq 1.\]
	 In \cite{MSS,SMS} the authors studied inequalities related to $w_p(\mathcal{T})$ for $\mathcal{T}\in L(H)^k.$ Motivated by the concept of joint numerical range (radius) of Hilbert space operators, we generalize these concepts on a  Banach space. Suppose $T_1,T_2,\ldots,T_k\in L(X,Y)$ and $\mathcal{T}=(T_1,T_2,\ldots,T_k)\in L(X,Y)^k.$  For $1\leq p<\infty,$ we define 
	\[\|\mathcal{T}\|_p=\sup\Bigg\{\Big(\sum_{i=1}^k\|T_ix\|^p\Big)^{1/p}:x\in S_X\Bigg\}.\]
	Using Minkowski's inequality, it is easy to observe that $\|\cdot\|_p$ actually defines a norm on $L(X,Y)^k.$ We call $\|\mathcal{T}\|_p$ as the  $p$-th joint operator norm of $\mathcal{T}.$  Similarly, we define the joint numerical range $W(\mathcal{T})$ of $\mathcal{T}\in L(X)^k$ as follows:
	\[W(\mathcal{T})=\Big\{(x^*(T_1x),x^*(T_2x),\ldots,x^*(T_kx)):(x,x^*)\in \Pi_X\Big\}.\]
	For $1\leq p<\infty,$ we define the $p$-th joint numerical radius $w_p(\mathcal{T})$ of $\mathcal{T}\in L(X)^k$ in the following way:
	\[w_p(\mathcal{T})=\sup\Bigg\{\Big(\sum_{i=1}^k |x^*(T_ix)|^p\Big)^{1/p}:(x,x^*)\in \Pi_X\Bigg\}.\] 
	Next, we generalize the notion of the numerical index of a Banach space. For $1\leq p<\infty,$ and $k\in \mathbb{N},$ we define the $(p,k)$-th joint numerical index of $X$ as follows:
	\[n_{(p,k)}(X)=\inf\{w_p(\mathcal{T}):\mathcal{T}\in L(X)^k,\|\mathcal{T}\|_p=1\}.\]
	Observe that, the $(p,k)$-th joint numerical index of a Banach space is monotonically decreasing with respect to $k.$ Indeed, if $k_1<k_2,$ then $L(X)^{k_1}\subset L(X)^{k_2} ,$ which implies that $n_{(p,k_1)}(X)\leq n_{(p,k_2)}(X).$\\
	Let us now briefly discuss the content of the article. After this introductory section,  in Section 2, we study the joint numerical range and the $p$-th joint numerical radius of a $k$-tuple of operators on a Banach space. We obtain a necessary condition for the convexity of $W(\mathcal{T}).$ Analogous to the numerical radius, we prove that in a finite-dimensional Banach space, $w_p(\mathcal{T})$ can be retrieved from a subset of $\Pi_X,$ namely from the set $G_X=\{(x,x^*)\in \Pi_X:x\in E_X,x^*\in E_{X^*}\}.$ Furthermore, we compare some basic properties of the operator norm and the numerical radius with that of the joint operator norm and the $p$-th joint numerical radius. In Section 4, we explore the $(p,k)$-th joint numerical index of  Banach spaces. In this direction, we first obtain a lower bound and an upper bound of $n_{(p,k)}(X)$ in terms of $n(X).$ Then we study $n_{(p,k)}(X),$ where $X$ is a direct sum of Banach spaces. As a consequence, we can explicitly compute $n_{(p,k)}(X),$ where $X$ is a direct sum of certain class of  Banach spaces. Finally, using the results of this section, we compute $n_{(p,k)}(X),$ where $X\in \{\ell_\infty,\ell_1,\ell_2,\ell_\infty^m,\ell_1^m,\ell_2^m:m\in \mathbb{N}\}.$

	\section{Joint numerical range and the $p$-th joint numerical radius}
	To begin with, we obtain a necessary condition for the convexity of the joint numerical range of an operator tuple on a Banach space.
	\begin{prop}\label{prop-convex}
		Let $\mathcal{T}=(T_1,T_2,\ldots,T_k)\in L(X)^k.$ If $W(\mathcal{T})$ is convex, then $W(T_i)$ is convex for each $1\leq i\leq k.$
	\end{prop}
    \begin{proof}
    	Let $W(\mathcal{T})$ be convex. We show that $W(T_i)$ is convex for all $1\leq i\leq k.$ Choose $\lambda_1,\lambda_2\in W(T_i)$ and  $t\in(0,1).$ Then there exist $(x_j,x_j^*)\in \Pi_X$ for $j=1,2,$ such that $\lambda_j=x_j^*(T_ix_j).$  Now, $\Lambda_j=\big(x_j^*(T_1x_j),x_j^*(T_2x_j),\ldots,x_j^*(T_kx_j)\big)\in W(\mathcal{T})$ for $j=1,2.$ Since $W(\mathcal{T})$ is convex, for all $t\in(0,1),$ $t\Lambda_1+(1-t)\Lambda_2\in W(\mathcal{T}).$ Thus, there exists $(x,x^*)\in \Pi_X$ such that 
    	\[t\Lambda_1+(1-t)\Lambda_2=\big(x^*(T_1x),x^*(T_2x),\ldots,x^*(T_kx)\big).\]
    	Hence, $tx_1^*(T_ix_1)+(1-t)x_2^*(T_ix_2)=x^*(T_ix)\in W(T_i).$ Therefore, $W(T_i)$ is convex. This completes the proof.
    \end{proof}
	The converse of Proposition \ref{prop-convex} is not true in general. From \cite{BD2}, it is known that $W(T)$ is always connected for $T\in L(X).$ Therefore, if $X$ is a real Banach space, then $W(T)$ is convex. However, it is easy to construct examples on real Banach spaces,  where the joint numerical range is not always convex. Next, we exhibit such an example.
	
	\begin{example}
		 Consider the real Banach space $X=\ell_1^2.$ Suppose $T,S\in L(X)$ be defined as $T(a,b)=(b,0),~S(a,b)=(0,a)$ for all $(a,b)\in X.$ Observe that $G_X=\{((1,0),(1,\pm1)),((0,1),(\pm 1,1))\}$ and
		 \[\Pi_X=G_X\cup\{((a,b),(sgn ~a,sgn ~b )):a\neq 0,b\neq 0,|a|+|b|=1\},\]
		 where $sgn$ is the usual sign function. Let $\mathcal{T}=(T,S).$ Then it is easy to check that
		 \[W(\mathcal{T})=\{\pm (a,b):0\leq a,b\leq 1,a+b=1\}.\]
		 Thus, $(-1,0),(0,1)\in W(\mathcal{T}).$ However,  $t(-1,0)+(1-t)(0,1)=(-t,1-t)\notin W(\mathcal{T})$ for all $t\in (0,1).$ This proves that $W(\mathcal{T})$ is not convex.	 
	\end{example}
	In the following example, we construct an example of $\mathcal{T}\in L(X)^k$ such that $W(\mathcal{T})$ is convex.
	\begin{example}
		Choose $T\in L(X)$ such that $W(T)$ is convex. Let $T_i=c_iI$ for all $1\leq i\neq j\leq k,$ where $c_i$ are scalars and $T_j=T.$ Consider $\mathcal{T}=(T_1,T_2,\ldots,T_k).$ Then clearly, $W(\mathcal{T})\subseteq \prod_{i=1}^k W(T_i).$ On the other hand, suppose $(a_1,a_2,\ldots,a_k)\in \prod_{i=1}^k W(T_i).$ Now, $a_j\in W(T_j)$ implies that there exists $(x,x^*)\in \Pi_X$ such that $x^*(T_jx)=a_j.$ Moreover, for all $1\leq i\neq j\leq k,$ since $W(T_i)=\{c_i\},$  $a_i=c_i=x^*(c_iIx)=x^*(T_ix).$ Thus, $(a_1,a_2,\ldots,a_k)=(x^*(T_1x),x^*(T_2x),\ldots,x^*(T_kx))\in W(\mathcal{T}).$ Therefore, $\prod_{i=1}^k W(T_i)\subseteq W(\mathcal{T}),$ i.e., $W(\mathcal{T})=\prod_{i=1}^k W(T_i).$ Now, the convexity of $W(\mathcal{T})$ follows from the fact that $\prod_{i=1}^k W(T_i)$ is convex. 
	\end{example} 
	Next, we find out a necessary and sufficient condition for the $p$-th joint numerical radius to define a norm on $L(X)^k.$
	\begin{prop}
		The $p$-th joint numerical radius, $w_p(\cdot)$ defines a norm on $L(X)^k$ if and only if the numerical radius $w(\cdot)$ defines a norm on $L(X).$
	\end{prop}
	\begin{proof}
		Suppose that $w(\cdot)$ defines a norm on $L(X).$ Let $\mathcal{T}=(T_1,T_2,\ldots,T_k)\in L(X)^k.$ Then 
		\begin{align*}
			&w_p(\mathcal{T})=0\\
			\Rightarrow & \sup\Bigg\{\Big(\sum_{i=1}^k |x^*(T_ix)|^p\Big)^{1/p}:(x,x^*)\in \Pi_X\Bigg\}=0\\
			\Rightarrow &\Big(\sum_{i=1}^k |x^*(T_ix)|^p\Big)^{1/p}=0, ~\forall~(x,x^*)\in \Pi_X\\
				\Rightarrow & |x^*(T_ix)|^p=0, ~\forall~(x,x^*)\in \Pi_X,~\forall ~1\leq i\leq k\\
				\Rightarrow &w(T_i)=0,~\forall ~1\leq i\leq k\\
				\Rightarrow & T_i=O, ~\forall ~1\leq i\leq k,~~~(\text{since~}w(\cdot)~\text{is~a~norm})\\
				\Rightarrow & \mathcal{T}=(O,O,\ldots,O).
		\end{align*}
	Now, let $a$ be a scalar. Then form $\Big(\sum_{i=1}^k |x^*(aT_ix)|^p\Big)^{1/p}=|a|\Big(\sum_{i=1}^k |x^*(T_ix)|^p\Big)^{1/p},$ it follows that $w_p(a\mathcal{T})=|a|w_p(\mathcal{T}).$ Next, let  $\mathcal{T}_j=(T_{1j},T_{2j},\ldots,T_{kj})\in L(X)^k,$ for $j=1,2,$ Using Minkowski's inequality, we get,
	\begin{align*}
	\Big(\sum_{i=1}^k |x^*(T_{i1}x)+x^*(T_{i2}x)|^p\Big)^{1/p}\leq \Big(\sum_{i=1}^k |x^*(T_{i1}x)|^p\Big)^{1/p}+\Big(\sum_{i=1}^k |x^*(T_{i2}x)|^p\Big)^{1/p}	.	
		\end{align*}
	In the above inequality, taking supremum over all $(x,x^*)\in \Pi_X,$ we get $w_p(\mathcal{T}_1+\mathcal{T}_2)\leq w_p(\mathcal{T}_1)+w_p(\mathcal{T}_2).$ Hence, $w_p(\cdot)$ defines a norm on $L(X)^k.$\\
	Conversely, suppose that $w_p(\cdot)$ defines a norm on $L(X)^k.$ If possible, suppose $w(\cdot)$ is not a norm on $L(X).$ Then there exists a non-zero operator $T\in L(X)$ such that $w(T)=0.$ Consider $\mathcal{T}=(T,O,\ldots,O)\in L(X)^k.$ Observe that, $w_p(\mathcal{T})=w(T)=0,$ whereas $\mathcal{T}$ is non-zero. This is a contradiction. Therefore, $w(\cdot)$ is a norm on $L(X).$ This completes the proof.
	\end{proof}
	
	Recall from \cite[Lem. 2.5]{MO} that for $T\in L(X),$ $$w(T)=\sup\{|x^{**}(T^*x^*)|: x^*\in E_{X^*},x^{**}\in E_{X^{**}},x^{**}(x^*)=1 \}.$$
	In the following theorem, we obtain an analogous result for the $p$-th joint numerical radius of operators on a finite-dimensional Banach space.
	\begin{theorem}\label{th-extreme}
		Let $\dim(X)<\infty.$ Then for $\mathcal{T}\in L(X)^k,$ 
		\[w_p(\mathcal{T})=\max\Bigg\{\Big(\sum_{i=1}^k |x^*(T_ix)|^p\Big)^{1/p}:(x,x^*)\in G_X\Bigg\},\]
		where $G_X=\{(x,x^*)\in \Pi_X:x\in E_X,x^*\in E_{X^*}\}.$
	\end{theorem}
\begin{proof}
Since $\dim(X)<\infty,$ $S_X,S_{X^*}$ are compact. Hence $S_X\times S_{X^*}$ is compact.  Observe that, $\Pi_X$ is a closed subset of $S_X\times S_{X^*}.$ Indeed, if $\{(x_n,x_n^*)\}$ is a sequence of $\Pi_X$ converging to $(x,x^*),$ then $x_n\to x$ and $x_n^*\to x^*.$ Clearly, $\|x\|=1,\|x^*\|=1.$ Thus,
\begin{align*}
	|1-x^*(x)|&=|x_n^*(x_n)-x^*(x)|\\
	                &\leq |x_n^*(x_n)-x_n^*(x)|+|x_n^*(x)-x^*(x)|\\
                    &\leq \|x_n^*\|\|x_n-x\|+\|x_n^*-x^*\|\|x\|		\\
                    &\to 0,
\end{align*} 
from which it follows that $x^*(x)=1,$ and so $(x,x^*)\in \Pi_X.$ This proves that $\Pi_X$ is a compact subset of $S_X\times S_{X^*}.$ Observe that, the mapping $\phi:X\times X^*\to F$ defined by $\phi(x,x^*)=\Big(\sum_{i=1}^k |x^*(T_ix)|^p\Big)^{1/p}$ is continuous. Therefore, the set $\phi(\Pi_X)=\Bigg\{\Big(\sum_{i=1}^k |x^*(T_ix)|^p\Big)^{1/p}:(x,x^*)\in \Pi_X\Bigg\}$ is compact. Hence, there exists $(x_0,x_0^*)\in \Pi_X$ such that $\phi(x_0,x_0^*)=\sup\phi(\Pi_X).$ Thus, $w_p(\mathcal{T})=\Big(\sum_{i=1}^k |x_0^*(T_ix_0)|^p\Big)^{1/p}.$ Suppose that $x_0\notin E_X.$ Then there exist $\lambda_i\in (0,1)$ and $x_i\in E_X$ for $1\leq i\leq n$ such that $\sum_{i=1}^n\lambda_i=1$ and $x_0=\sum_{i=1}^n\lambda_ix_i.$ Now,
\begin{align*}
T_jx_0&=	\sum_{i=1}^n\lambda_iT_jx_i~\forall ~1\leq j\leq k\\
\Rightarrow |x_0^*(T_jx_0)|   &= |\sum_{i=1}^n\lambda_ix_0^*(T_jx_i)|~\forall ~1\leq j\leq k\\
	                                              &\leq \sum_{i=1}^n\lambda_i|x_0^*(T_jx_i)|~\forall ~1\leq j\leq k\\
	  \Rightarrow |x_0^*(T_jx_0)| ^p&\leq \bigg(\sum_{i=1}^n\lambda_i|x_0^*(T_jx_i)|\bigg)^p~\forall ~1\leq j\leq k\\
	  &\leq \sum_{i=1}^n\lambda_i |x_0^*(T_jx_i)|^p ~\forall ~1\leq j\leq k\\
	  \Rightarrow \sum_{j=1}^k |x_0^*(T_jx_0)| ^p&\leq \sum_{i=1}^n\lambda_i \bigg(\sum_{j=1}^k|x_0^*(T_jx_i)|^p\bigg)\\
	  &\leq \sum_{i=1}^n\lambda_i w_p(\mathcal{T})^p=w_p(\mathcal{T})^p.
\end{align*}
Thus, from
\[w_p(\mathcal{T})^p=\sum_{j=1}^k |x_0^*(T_jx_0)| ^p\leq \sum_{i=1}^n\lambda_i \bigg(\sum_{j=1}^k|x_0^*(T_jx_i)|^p\bigg)\leq w_p(\mathcal{T})^p,\]
it follows that for all $1\leq i\leq n,$ $w_p(\mathcal{T})^p=\bigg(\sum_{j=1}^k|x_0^*(T_jx_i)|^p\bigg).$ Moreover, from 
\[1=x_0^*(x_0)=\sum_{i=1}^n\lambda_ix_0^*(x_i)=|\sum_{i=1}^n\lambda_ix_0^*(x_i)|\leq \sum_{i=1}^n\lambda_i|x_0^*(x_i)|\leq \sum_{i=1}^n\lambda_i=1,\]
we get $x_0^*(x_i)=1$ for all $1\leq i\leq n, $ i.e., $(x_i,x_0^*)\in \Pi_X.$ Now, suppose that $x_0^*\notin E_{X^*}.$ Then there exist $\mu_i\in(0,1)$ and $x_i^*\in E_{X^*}$ for $1\leq i\leq m,$ such that $\sum_{i=1}^m\mu_i=1$ and $x_0^*=\sum_{i=1}^m\mu_ix_i^*.$ Following similar arguments, we can show that for all $1\leq i\leq n,1\leq r\leq m,$ $w_p(\mathcal{T})=\bigg(\sum_{j=1}^k|x_r^*(T_jx_i)|^p\bigg)^{1/p},$ where $(x_i,x_r^*)\in \Pi_X.$ Moreover, since $x_i\in E_X,x_r^*\in E_{X^*},$ we have $(x_i,x_r^*)\in G_X.$ This completes the proof.
\end{proof}
We can obtain an analogous result of Theorem \ref{th-extreme} for $\|\mathcal{T}\|_p,$ where $\mathcal{T}\in L(X,Y)^k.$ To avoid monotonicity, we simply state the theorem.

\begin{theorem}\label{th-extremenorm}
Let $\dim(X)<\infty.$ Then for $\mathcal{T}\in L(X,Y)^k,$ 
\[\|\mathcal{T}\|_p=\max\Bigg\{\Big(\sum_{i=1}^k \|T_ix\|^p\Big)^{1/p}:x\in E_X\Bigg\}.\]
\end{theorem}

For $\mathcal{T}=(T_1,\ldots,T_k)\in L(X,Y)^k,$ suppose $\mathcal{T}^*=(T_1^*,\ldots,T_k^*)\in L(Y^*,X^*)^k.$ In the next example, we exhibit a basic difference between the operator norm and the $p$-th joint operator norm. More precisely, we show that  in general, for $k>1,$ $\|\mathcal{T}\|_p\neq \|\mathcal{T}^*\|_p.$
\begin{example}
	Suppose $m>2.$ Choose $1<k\leq m.$ Consider the operators $T_i\in L(\ell_\infty^m)$ for $1\leq i\leq k$ as follows:
	\[T_ie_i=e_i, ~T_ie_j=\theta,~\text{for}~j\in \{1,\ldots,m\}\setminus \{i\},\] where $\{e_i:1\leq i\leq m\}$ is the standard ordered basis of $\ell_\infty^m.$ Suppose $\mathcal{T}=(T_1,\ldots,T_k)\in L(\ell_\infty^m)^k.$ Then $\|\mathcal{T}\|_p=k^{1/p}$ (see Theorem \ref{th-l1} (i) for details).\\
	Now, observe that, $\mathcal{T}^*=(T_1,\ldots,T_k)\in L(\ell_1^m)^k,$ since $T_i^*=T_i$ for all $1\leq i\leq k.$ Therefore, using Theorem \ref{th-extremenorm}, we get
	\[\|\mathcal{T}^*\|_p=\sup\Big\{\big(\sum_{i=1}^k\|T_ie_j\|^p\big)^{1/p}:1\leq j\leq m\Big\}=1\neq \|\mathcal{T}\|_p.\]
\end{example}
However, if $\mathcal{T}\in L(X)^k,$ where $X$ is a reflexive Banach space, then $w_p(\mathcal{T})=w_p(\mathcal{T^*}).$
\begin{prop}
For $\mathcal{T}=(T_1,\ldots,T_k)\in L(X)^k,$ $w_p(\mathcal{T})\leq w_p(\mathcal{T}^*).$ In particular, if $X$ is reflexive, then $w_p(\mathcal{T})=w_p(\mathcal{T}^*).$ 
\end{prop}
\begin{proof}
	Suppose $Q:X\to X^{**}$ is the natural isometric embedding of $X$ into $X^{**}$. Observe that, if $(x,x^*)\in \Pi_X,$ then $Q(x)\in S_{X^{**}}$ and $Q(x)(x^*)=x^*(x)=1.$ Therefore, $(x^*,Q(x))\in \Pi_{X^*}.$ Now,
	\begin{align*}
		w_p(\mathcal{T})&=\sup\Bigg\{\Big(\sum_{i=1}^k |x^*(T_ix)|^p\Big)^{1/p}:(x,x^*)\in \Pi_X\Bigg\}\\
		                              &=\sup\Bigg\{\Big(\sum_{i=1}^k |Q(x)(T_i^*x^*)|^p\Big)^{1/p}:(x^*,Q(x))\in \Pi_{X^{*}}\Bigg\}\\
		                              &\leq \sup\Bigg\{\Big(\sum_{i=1}^k |x^{**}(T_i^*x^*)|^p\Big)^{1/p}:(x^*,x^{**})\in \Pi_{X^{*}}\Bigg\}\\
		                              &=	w_p(\mathcal{T}^*).
	\end{align*}
In particular, if $X$ is reflexive, then since $Q(X)=X^{**},$ the above inequality is an equality. Therefore, in this case 	$w_p(\mathcal{T})=	w_p(\mathcal{T}^*).$ This ends the proof.
\end{proof}

\section{The $(p,k)$-th joint numerical index}
Let us  begin this section with a bound of the $(p,k)$-th joint numerical index of a Banach space. We follow the idea of \cite[Th. 9]{DM} in the proof.  
\begin{theorem}\label{th-boundnpkx}
	For all $1\leq p<\infty$ and $k\in \mathbb{N},$ the following is true.
\[\frac{n(X)}{k^{1/p}}\leq n_{(p,k)}(X)\leq n(X).\]	
\end{theorem}	
	\begin{proof}
		To prove $ n_{(p,k)}(X)\leq n(X),$ consider an arbitrary non-zero operator $T\in L(X).$ Let $\mathcal{T}=(T,O,\ldots,O)\in L(X)^k.$ Then observe that, $\|\mathcal{T}\|_p=\|T\|$ and $w_p(\mathcal{T})=w(T).$ Thus, $n_{(p,k)}(X)\leq \frac{1}{\|\mathcal{T}\|_p}w_p(\mathcal{T})=\frac{1}{\|T\|}w(T).$   Since this is true for each non-zero operator $T\in L(X),$ we get 
		\begin{equation}\label{eq-b01}
			n_{(p,k)}(X)\leq n(X).
			\end{equation}
		Observe that, if $n(X)=0,$ then from (\ref{eq-b01}), we get the desired result. Therefore, 
		to prove the first inequality, without loss of generality, we may assume that $n(X)>0.$ Now, consider $\mathcal{T}=(T_1,T_2,\ldots,T_k)\in L(X)^k$ such that $\|\mathcal{T}\|_p=1.$ Choose $\epsilon>0.$ Then there exists $x\in S_{X}$ such that $(\sum_{i=1}^k\|T_ix\|^p)^{1/p}>1-\epsilon.$ Clearly, there exists $j_0\in \{1,2,\ldots,k\}$ such that $\|T_{j_0}x\|^p>\frac{(1-\epsilon)^p}{k}.$ Thus,
		\[w(T_{j_0})\geq n(X)\|T_{j_0}\|\geq n(X)\|T_{j_0}x\|>\frac{(1-\epsilon)}{k^{1/p}}n(X).\]
		Now,
		\begin{align*}
			w_p(\mathcal{T})&\geq \Big(\sum_{i=1}^k |x^*(T_ix)|^p\Big)^{1/p}, ~\forall~(x,x^*)\in \Pi_X\\
			& \geq |x^*(T_{j_0}x)|, ~\forall~(x,x^*)\in \Pi_X\\
			\Rightarrow w_p(\mathcal{T})&\geq w(T_{j_0})>\frac{(1-\epsilon)}{k^{1/p}}n(X).
		\end{align*}
	Since the above inequality is true for each $\epsilon>0,$ we have $w_p(\mathcal{T})\geq \frac{n(X)}{k^{1/p}}.$
	Thus, $n_{{p,k}}(X)\geq \frac{n(X)}{k^{1/p}}.$
		This completes the proof.
	\end{proof}
Note that, if $n(X)>0,$ then for each non-zero $T\in L(X), $ $w(T)>0.$ However, from \cite[Ex. 3b]{MP}, it follows that the converse is not true. In particular, there is a Banach space $X,$ for which $n(X)=0$ but $w(T)>0$ for all non-zero operator $T.$ Using this fact and the last theorem, we immediately get an analogous result for the $(p,k)$-th joint numerical index of a Banach space.
\begin{cor}
There exists a Banach space $X$ such that $n_{(p,k)}(X)=0,$ whereas $w_p(\mathcal{T})\neq 0$ for all non-zero $\mathcal{T}\in L(X)^k.$ 	
\end{cor}
\begin{proof}
Using \cite[Ex. 3b]{MP}, we can choose a Banach space $X$ such that $n(X)=0,$ whereas $w(T)\neq 0$ for all non-zero $T\in L(X).$ Now, suppose $\mathcal{T}=(T_1,\ldots,T_k)\in L(X)^k$ and $\mathcal{T}$ is non-zero. Then $T_{i_0}$ is non-zero for some $1\leq i_0\leq k.$ Therefore, $w_p(\mathcal{T})\geq w(T_{i_0})>0.$ Whereas, from Theorem \ref{th-boundnpkx}, we get 
$$0=\frac{n(X)}{k^{1/p}}\leq n_{(p,k)}(X)\leq n(X)=0,$$ which implies that $n_{(p,k)}(X)=0.$ 	
\end{proof}
Our next goal is to study the $(p,k)$-th joint numerical index of direct sum of Banach spaces. Suppose $\{X_\lambda:\lambda\in \Lambda\}$ is a family of Banach spaces. By the symbol $[\oplus_{\lambda\in \Lambda} X_{\lambda}]_{\ell_q},$ where $1\leq q\leq \infty$  (resp. $[\oplus_{\lambda\in \Lambda} X_{\lambda}]_{c_0}$), we denote the $\ell_q$-sum (resp. $c_0$-sum) of the family. In what follows, we use the notation $J(x)$ to denote the set $\{x^*\in S_{X^*}:x^*(x)=\|x\|\}$ for all non-zero $x\in X.$
\begin{theorem}\label{th-directsum}
Let $\{X_\lambda:\lambda\in \Lambda\}$ be a family of Banach spaces. Suppose $X$ is either $[\oplus_{\lambda\in \Lambda} X_{\lambda}]_{\ell_q},$ where $1\leq q\leq \infty$ or  $[\oplus_{\lambda\in \Lambda} X_{\lambda}]_{c_0}.$ Then
\[n_{(p,k)}(X)\leq \inf\{n_{(p,k)}(X_{\lambda}):\lambda\in \Lambda\}.\] 	
\end{theorem}
\begin{proof}
	Suppose that $Z=[Y\oplus V]_{\ell_q}$ for the Banach spaces $Y,V,$ where $1\leq q\leq \infty.$ We first claim that $n_{(p,k)}(Z)\leq n_{(p,k)}(Y).$ Choose $\mathcal{T}_Y=(A_1,A_2,\ldots,A_k)\in L(Y)^k$ such that $\|\mathcal{T}_Y\|_p=1.$ For $1\leq i\leq k,$ define $T_i\in L(Z)$ by $T_i(y,v)=(A_i y,\theta),$ where $y\in Y,v\in V.$ Let $\mathcal{T}_Z=(T_1,T_2,\ldots,T_k)\in L(Z)^k.$ Observe that,
	\begin{align*}
		\|\mathcal{T}_Z\|_p&=\sup_{(y,v)\in S_Z}\Big(\sum_{i=1}^k\|T_i(y,v)\|^p\Big)^{1/p}\\
		&=\sup_{(y,v)\in S_Z}\Big(\sum_{i=1}^k\|A_iy\|^p\Big)^{1/p}\\
		&\geq \sup_{(y,\theta)\in S_Z}\Big(\sum_{i=1}^k\|A_iy\|^p\Big)^{1/p}\\
		&=\sup_{y\in S_Y}\Big(\sum_{i=1}^k\|A_iy\|^p\Big)^{1/p}\\
		&=\|\mathcal{T}_Y\|_p.
	\end{align*}
Thus,
\begin{equation}\label{eq-01}
	\|\mathcal{T}_Z\|_p\geq \|\mathcal{T}_Y\|_p.
	\end{equation}
On the other hand, suppose $(y,v)\in S_Z.$  Then $\|y\|\leq 1.$ Now, for $y=\theta,$ we have, $\Big(\sum_{i=1}^k\|A_iy\|^p\Big)^{1/p}=0\leq \|\mathcal{T}_Y\|_p$ and for $y\neq \theta,$ $$\Big(\sum_{i=1}^k\|A_iy\|^p\Big)^{1/p}= \|y\|\Big(\sum_{i=1}^k\|A_i\frac{y}{\|y\|}\|^p\Big)^{1/p}\leq \Big(\sum_{i=1}^k\|A_i\frac{y}{\|y\|}\|^p\Big)^{1/p}\leq\|\mathcal{T}_Y\|_p.$$
Thus, for each $(y,v)\in S_Z,$ $\Big(\sum_{i=1}^k\|T_i(y,v)\|^p\Big)^{1/p}=\Big(\sum_{i=1}^k\|A_iy\|^p\Big)^{1/p}\leq\|\mathcal{T}_Y\|_p.$
This implies that 
\begin{equation}\label{eq-02}
	\|\mathcal{T}_Z\|_p\leq \|\mathcal{T}_Y\|_p.
\end{equation}
Now, from (\ref{eq-01}) and (\ref{eq-02}), it follows that
	$\|\mathcal{T}_Z\|_p= \|\mathcal{T}_Y\|_p=1.$ \\
	Next, we show that $w_p(\mathcal{T}_Z)=w_p(\mathcal{T}_Y).$
 Note that, $(y,y^*)\in \Pi_Y$ is equivalent to $((y,\theta),(y^*,\theta))\in \Pi_Z.$ Now,
	\begin{align*}
		w_p(\mathcal{T}_Z)&=\sup_{((y,v),(y^*,v^*))\in \Pi_Z}\Big(\sum_{i=1}^k|(y^*,v^*)(T_i(y,v))|^p\Big)^{1/p}\\
	&=\sup_{((y,v),(y^*,v^*))\in \Pi_Z}\Big(\sum_{i=1}^k|y^*(A_iy)|^p\Big)^{1/p}\\
	&\geq \sup_{((y,\theta),(y^*,\theta))\in \Pi_Z}\Big(\sum_{i=1}^k|y^*(A_iy)|^p\Big)^{1/p}\\
	&=\sup_{(y,y^*)\in \Pi_Y}\Big(\sum_{i=1}^k|y^*(A_iy)|^p\Big)^{1/p}\\
	&=w_p(\mathcal{T}_Y).
	\end{align*}  
Thus,
\begin{equation}\label{eq-03}
w_p(\mathcal{T}_Z)	\geq w_p(\mathcal{T}_Y).
\end{equation}
On the other hand, let $((y,v),(y^*,v^*))\in \Pi_Z.$ Then $\|y\|\leq 1,$ $\|y^*\|\leq 1.$ Now, suppose that $q=\infty.$ Then from \cite[Prop. 6.1]{CKS}, it follows that $\|y^*\|+\|v^*\|=1$ and the following three cases may hold.\\
(i) $\|y\|<\|v\|,$ (ii) $\|y\|=\|v\|,$ (iii) $\|y\|>\|v\|.$ \\
 For case (i), by \cite[Prop. 6.1]{CKS}, we get $y^*=\theta.$ Hence, 
\[\Big(\sum_{i=1}^k|(y^*,v^*)(T_i(y,v))|^p\Big)^{1/p}=\Big(\sum_{i=1}^k|y^*(A_iy)|^p\Big)^{1/p}=0\leq w_p(\mathcal{T}_Y).\]
Note that, for cases (ii) and (iii), if $y^*=\theta$ holds, then  similarly, we get the above inequality. Therefore, consider $y^*\neq \theta.$  For case (ii), i.e., if $\|y\|=\|v\|$ holds, then clearly, $\|y\|\neq 0$ and 
\[1=(y^*,v^*)(y,v)=y^*(y)+v^*(v)\leq \|y^*\|\|y\|+\|v^*\|\|v\|=(\|y^*\|+\|v^*\|)\|y\|\leq 1.\]
Thus, the above inequality is actually an equality and so $y^*(y)=\|y^*\|\|y\|.$ For case (iii), clearly, $\|y\|\neq 0$ and once again using \cite[Prop. 6.1]{CKS}, we get $y^*\in J(y),$ i.e., $y^*(y)=\|y^*\|\|y\|.$ Therefore, if case (ii) or case (iii) holds, then 
	\begin{align*}
&\Big(\sum_{i=1}^k|(y^*,v^*)(T_i(y,v))|^p\Big)^{1/p}\\
	&=\Big(\sum_{i=1}^k|y^*(A_iy)|^p\Big)^{1/p}\\
	&=\|y^*\|\|y\|\Big(\sum_{i=1}^k\Big|\frac{y^*}{\|y^*\|}\Big(A_i\frac{y}{\|y\|}\Big)\Big|^p\Big)^{1/p}\\
	&\leq w_p(\mathcal{T}_Y).
\end{align*}  
Thus, for each case, $\Big(\sum_{i=1}^k|(y^*,v^*)(T_i(y,v))|^p\Big)^{1/p}\leq w_p(\mathcal{T}_Y).$ Since the last inequality holds for each $((y,v),(y^*,v^*))\in \Pi_Z,$ we get 
\begin{equation}\label{eq-04}
w_p(\mathcal{T}_Z)\leq w_p(\mathcal{T}_Y).
\end{equation}
Thus, from (\ref{eq-03}) and (\ref{eq-04}), it follows that if $q=\infty,$ then $w_p(\mathcal{T}_Z)=w_p(\mathcal{T}_Y).$ Similarly, for $1\leq q<\infty,$ we can show that if $((y,v),(y^*,v^*))\in \Pi_Z,$ then either $y=\theta$ or $0<y^*(y)=\|y^*\|\|y\|.$ Now, using similar arguments, we can prove that $$w_p(\mathcal{T}_Z)=w_p(\mathcal{T}_Y).$$
Thus, for each $\mathcal{T}_Y\in L(Y)^k$ with $\|\mathcal{T}_Y\|_p=1,$
\[n_{(p,k)}(Z)\leq w_p(\mathcal{T}_Z)=w_p(\mathcal{T}_Y).\]
Therefore, it follows that $n_{(p,k)}(Z)\leq n_{(p,k)}(Y)$ and our claim is proved.\\
Now, fix $\lambda_0\in \Lambda.$ For $X=[\oplus_{\lambda\in \Lambda} X_{\lambda}]_{c_0},$ assume $V=[\oplus_{\lambda\in \Lambda,\lambda\neq \lambda_0} X_{\lambda}]_{c_0}.$ Then we can write  $$X=[X_{\lambda_0}\oplus V]_{\ell_\infty},$$ and so $n_{(p,k)}(X)\leq n_{(p,k)}(X_{\lambda_0}).$ Similarly, if $X=[\oplus_{\lambda\in \Lambda} X_{\lambda}]_{\ell_q},$ where $1\leq q\leq \infty,$ then  assuming $V=[\oplus_{\lambda\in \Lambda,\lambda\neq \lambda_0} X_{\lambda}]_{\ell_q},$  writing  $$X=[X_{\lambda_0}\oplus V]_{\ell_q},$$ we get $n_{(p,k)}(X)\leq n_{(p,k)}(X_{\lambda_0}).$ This completes the proof of the theorem.
\end{proof}
We would like to remark here that in \cite{MP}, Mart\'in and Pay\'a studied the numerical index of direct sum of Banach spaces. They proved the following result.
\begin{theorem}\cite[Prop. 1 \& Rem. 2a]{MP}\label{th-mp}
Let $\{X_\lambda:\lambda\in \Lambda\}$ be a family of Banach spaces.  Then 
\[n([\oplus_{\lambda\in \Lambda} X_{\lambda}]_{c_0})=n([\oplus_{\lambda\in \Lambda} X_{\lambda}]_{\ell_1})=n([\oplus_{\lambda\in \Lambda} X_{\lambda}]_{\ell_\infty})=\inf_{\lambda\in \Lambda}n(X_\lambda).\]
If $1<q<\infty,$ then 
\[n([\oplus_{\lambda\in \Lambda} X_{\lambda}]_{\ell_q})\leq\inf_{\lambda\in \Lambda}n(X_\lambda).\]
\end{theorem}
As an immediate consequence of Theorem \ref{th-directsum} and Theorem \ref{th-mp}, we can compute the $(p,k)$-th joint numerical index of a direct sum of certain class of Banach spaces. 
\begin{cor}\label{cor-direct}
	Suppose $\{X_\lambda:\lambda\in \Lambda\}$ is a family of Banach spaces. Let $$X\in \Bigg\{[\oplus_{\lambda\in \Lambda} X_{\lambda}]_{\ell_1},[\oplus_{\lambda\in \Lambda} X_{\lambda}]_{\ell_\infty}, [\oplus_{\lambda\in \Lambda} X_{\lambda}]_{c_0}\Bigg\}. $$ Then the following are true.\\
	(i) If $n(X)=1$ and $n_{(p,k)}(X_{\lambda_0})=\frac{1}{k^{1/p}}$ for some $\lambda_0\in \Lambda,$ then $n_{(p,k)}(X)=\frac{1}{k^{1/p}}.$\\
	(ii) If $n_{(p,k)}(X_\lambda)=\frac{n(X_\lambda)}{k^{1/p}}$ for each $\lambda  \in \Lambda,$  then $n_{(p,k)}(X)=\frac{n(X)}{k^{1/p}}.$
\end{cor}
\begin{proof}
	(i) Note that,
	\[\frac{1}{k^{1/p}}=\frac{n(X)}{k^{1/p}}\leq n_{(p,k)}(X)\leq \inf\{n_{(p,k)}(X_\lambda):\lambda\in \Lambda\}\leq n_{(p,k)}(X_{\lambda_0})=\frac{1}{k^{1/p}},\]
	where the first inequality follows from Theorem \ref{th-boundnpkx} and the second inequality follows from Theorem \ref{th-directsum}. Thus, the above inequalities are actually equalities and so $n_{(p,k)}(X)=\frac{1}{k^{1/p}}.$\\
	(ii) Note that, 
	\[\frac{n(X)}{k^{1/p}}\leq n_{(p,k)}(X)\leq \inf\{n_{(p,k)}(X_\lambda):\lambda\in \Lambda\}= \inf\Big\{\frac{n(X_{\lambda})}{k^{1/p}}:\lambda\in \Lambda\Big\}=\frac{n(X)}{k^{1/p}},\]
	where the last equality follows from Theorem \ref{th-mp}. This completes the proof of (ii).
\end{proof}
\begin{remark}
From Corollary \ref{cor-direct} (i), it follows that the class of Banach spaces with the $(p,k)$-th joint numerical index $\frac{1}{k^{1/p}}$ is stable under $\ell_\infty$-sum (or $\ell_1$-sum or $c_0$-sum), if the sum has numerical index $1.$  	
\end{remark}
In the next theorem, we get another upper bound for $\ell_q$-sum $(1<q<\infty)$ of Banach spaces.
\begin{theorem}\label{th-ndirect}
Let $X,Y$ be Banach spaces. Suppose $Z=[X\oplus Y]_{\ell_q},$ where $q\in(1,\infty)$ and $\frac{1}{q}+\frac{1}{q'}=1.$ Then $n_{(p,k)}(Z)\leq \max\{\frac{1}{q},\frac{1}{q'}\}.$ 	
\end{theorem}
\begin{proof}
	Choose $\mathcal{S}=(S_1,\ldots,S_k)\in L(Y,X)^k$ such that $\|\mathcal{S}\|_p=1.$ For $1\leq i\leq k,$ define $T_i\in L(Z)$ by $T_i(x,y)=(S_iy,\theta)$ for all $(x,y)\in Z.$ Suppose $\mathcal{T}=(T_1,\ldots,T_k)\in L(Z)^k.$ Proceeding as Theorem \ref{th-directsum}, we can show that $\|\mathcal{T}\|_p=\|\mathcal{S}\|_p=1.$ Now, suppose that $((x,y),(x^*,y^*))\in \Pi_Z.$ Then from \cite[Prop. 6.1]{CKS}, it follows that either of the following is true.\\
	(i) $x=x^*=\theta,$ (ii) $y=y^*=\theta,$ (iii) $x\neq \theta,y\neq \theta, x^*=\|x\|^{q-1}f, y^*=\|y\|^{q-1}g,$ where $f\in J(x)$ and $g\in J(y).$\\
	If (i) holds, then $$\Big(\sum_{i=1}^k|(\theta,y^*)(T_i(\theta,y))|^p\Big)^{1/p}=\Big(\sum_{i=1}^k|(\theta,y^*)(S_iy,\theta)|^p\Big)^{1/p}=0.$$
	If (ii) holds, then $$\Big(\sum_{i=1}^k|(x^*,\theta)(T_i(x,\theta))|^p\Big)^{1/p}=\Big(\sum_{i=1}^k|(x^*,\theta)(S_i\theta,\theta)|^p\Big)^{1/p}=0.$$
	If (iii) holds, then 
	\begin{align*}
		\Big(\sum_{i=1}^k|(x^*,y^*)(T_i(x,y))|^p\Big)^{1/p}&=\Big(\sum_{i=1}^k|(\|x\|^{q-1}f,\|y\|^{q-1}g)(S_iy,\theta)|^p\Big)^{1/p}\\
		&=\Big(\sum_{i=1}^k\|x\|^{p(q-1)}|f(S_iy)|^p\Big)^{1/p}\\
		&\leq\|x\|^{q-1}\Big(\sum_{i=1}^k\|S_iy\|^p\Big)^{1/p}\\
		&\leq\|x\|^{q-1}\|\mathcal{S}\|_p\|y\|\\
		&=\|x\|^{q-1}\|y\|\\
		&\leq \frac{\|y\|^q}{q}+\frac{\|x\|^{q'(q-1)}}{q'}~(\text{using~Young's~inequality})\\
		&=\frac{\|y\|^q}{q}+\frac{\|x\|^{q}}{q'}\\
		&\leq \max\{\frac{1}{q},\frac{1}{q'}\} (\|x\|^q+\|y\|^q)\\
		&=\max\{\frac{1}{q},\frac{1}{q'}\}, ~(\text{since}~ (x,y)\in S_Z).
	\end{align*}
Therefore, $w_p(\mathcal{T})\leq \max\{\frac{1}{q},\frac{1}{q'}\}.$ Now, it follows that 
\[n_{(p,k)}(Z)\leq w_p(\mathcal{T})\leq \max\{\frac{1}{q},\frac{1}{q'}\},\]
completing the proof of the theorem.
\end{proof}
Combining Theorem \ref{th-directsum} and Theorem \ref{th-ndirect}, we get the following inequality for the $(p,k)$-th joint numerical index of $\ell_q$-sum $(1<q<\infty)$ of Banach spaces.
\begin{cor}\label{cor-ndirect}
	Let $\{X_\lambda:\lambda\in \Lambda\}$ be a family of Banach spaces. Suppose  $X=[\oplus_{\lambda\in \Lambda} X_{\lambda}]_{\ell_q},$ where $1< q< \infty.$  Let $\frac{1}{q}+\frac{1}{q'}=1.$ Then
	\[n_{(p,k)}(X)\leq \inf\{\max\{\frac{1}{q},\frac{1}{q'}\},n_{(p,k)}(X_{\lambda}):\lambda\in \Lambda\}.\] 	
\end{cor}
Note that, if $k=1,$ then for all $1\leq p<\infty,~n_{(p,k)}(X)=n(X).$ Therefore, in Corollary \ref{cor-ndirect}, if we choose $X_\lambda$ such that $\max\{\frac{1}{q},\frac{1}{q'}\}<\inf \{n(X_\lambda):\lambda\in \Lambda\},$ then Corollary \ref{cor-ndirect} improves on both Theorem \ref{th-directsum} and Theorem \ref{th-mp}.
Finally, our aim is to use the results of this section to explicitly compute the $(p,k)$-th joint numerical index of some classical Banach spaces.
\begin{theorem}\label{th-l1}
 Suppose $m\in \mathbb{N}\setminus \{1\}.$ Then \\
(i) $n_{(p,k)}(\ell_\infty^m)=\frac{1}{k^{1/p}},$ if $1\leq k\leq m.$\\
(ii) $n_{(p,k)}(\ell_\infty)=\frac{1}{k^{1/p}}.$\\
(iii) $n_{(p,2)}(\ell_1^m)=\frac{1}{2^{1/p}}.$\\
(iv)  $n_{(p,2)}(\ell_1)=\frac{1}{2^{1/p}}.$
\end{theorem}
\begin{proof}
(i) Let $1\leq k\leq m.$ First we prove that $n_{(p,k)}(\ell_\infty^m)=\frac{1}{k^{1/p}}.$ It is well-known that $n(\ell_\infty^m)=1.$ Now, from Theorem \ref{th-boundnpkx}, it follows that 
\begin{equation}\label{eq-05}
\frac{1}{k^{1/p}}\leq n_{(p,k)}(\ell_\infty^m).	
	\end{equation}
 To prove the reverse inequality, define $T_i\in L(\ell_\infty^m)$ for $1\leq i\leq k$ as follows:
\[T_ie_i=e_i, ~T_ie_j=\theta,~\text{for}~j\in \{1,\ldots,m\}\setminus \{i\},\] where $\{e_i:1\leq i\leq m\}$ is the standard ordered basis of $\ell_\infty^m.$ Suppose $\mathcal{T}=(T_1,\ldots,T_k)\in L(\ell_\infty^m)^k.$ Clearly, $x=(a_1,a_2,\ldots,a_m)\in E_{\ell_\infty^m}$ if and only if $|a_i|=1,$ for all $1\leq i\leq m.$ Note that, $(x,x^*)\in G_{\ell_\infty^m}$ if and only if $x^*\in \{a_1e_1,a_2e_2,\ldots,a_me_m\}.$ Thus, for $1\leq i\leq k,1\leq j\leq m,$ and $x^*=a_je_j,$ 
\[x^*(T_ix)=a_je_j(a_ie_i)=\begin{cases}
	&0, ~~j\neq i\\
	                &1, ~~j=i.
	\end{cases}
\]
Now,  using Theorem \ref{th-extremenorm}, we get
\begin{align*}
\|\mathcal{T}\|_p&=\sup\Big\{\big(\sum_{i=1}^{k}\|T_i(a_1,a_2,\ldots,a_m)\|^p\big)^{1/p}:(a_1,a_2,\ldots,a_m)\in E_{\ell_\infty^m}\Big\}\\
&=\sup \Big\{\big(\sum_{i=1}^{k}|a_i|^p\big)^{1/p}:|a_i|=1,~\forall ~1\leq i\leq k\Big\}\\
&=k^{1/p}.
\end{align*}
On the other hand, using Theorem \ref{th-extreme}, we get
\begin{align*}
w_p(\mathcal{T})&=\max_{(x,x^*)\in G_{\ell_\infty^m}}\Big(\sum_{i=1}^k|x^*(T_ix)|^p\Big)^{1/p}\\
&=\max\Big\{\Big(\sum_{i=1}^k|a_je_j(a_ie_i))|^p\Big)^{1/p}:|a_j|=1,~\forall~1\leq j\leq m\Big\}	\\
&=1.
\end{align*}
Therefore,
\begin{equation}\label{eq-06}
	n_{(p,k)}(\ell_\infty^m)\leq \frac{w_p(\mathcal{T})}{\|\mathcal{T}\|_p}=\frac{1}{k^{1/p}}.
	\end{equation} 
Now, from (\ref{eq-05}) and (\ref{eq-06}), it follows that $n_{(p,k)}(\ell_\infty^m)=\frac{1}{k^{1/p}}.$\\

(ii) It is well-known that $n(\ell_\infty)=1.$ Hence, by Theorem \ref{th-boundnpkx},
\[\frac{1}{k^{1/p}}=\frac{n(\ell_\infty)}{k^{1/p}}\leq n_{(p,k)}(\ell_\infty).\]
For the reverse inequality, choose $m\geq k.$ Then we can write $\ell_\infty=[\ell_\infty^m\oplus\ell_\infty]_{\ell_\infty}.$ Therefore, by (i) and Theorem \ref{th-directsum}, we get $$n_{(p,k)}(\ell_\infty)\leq n_{(p,k)}(\ell_\infty^m)=\frac{1}{k^{1/p}}.$$
Thus, $n_{(p,k)}(\ell_\infty)=\frac{1}{k^{1/p}}.$\\

(iii) Observe that, for $m=2,$ the result follows from the fact that $\ell_\infty^2$ is isometrically isomorphic to $\ell_1^2.$ Suppose that $m>2.$ Then $\ell_1^m=[\ell_1^2\oplus\ell_1^{m-2}]_{\ell_1}.$ Therefore, from Theorem \ref{th-directsum}, we get $$n_{(p,2)}(\ell_1^m)\leq n_{(p,2)}(\ell_1^2)=\frac{1}{2^{1/p}}.$$ On the other hand, from Theorem \ref{th-boundnpkx}, it follows that  $$\frac{1}{2^{1/p}}=\frac{n(\ell_1^m)}{2^{1/p}}\leq n_{(p,2)}(\ell_1^m).$$
This completes the proof of (iii).\\

(iv) Note that, $\ell_1=[\ell_1^2\oplus\ell_1]_{\ell_1}.$ Now, proceeding similarly as (iii), we get the desired result.
\end{proof}
We end the section by computing the $(p,k)$-th joint numerical index of Hilbert spaces. 
\begin{theorem}\label{th-l2}
Suppose $m\in \mathbb{N}\setminus\{1,2\}. $ Let $\ell_2^m$ and $\ell_2$ be complex Hilbert spaces. Then \\
(i) $n_{(2,k)}(\ell_2^m)=\frac{1}{2\sqrt{k}},$ where $1\leq k<m.$	\\
(ii) $n_{(2,k)}(\ell_2)=\frac{1}{2\sqrt{k}}.$\\
(iii) $n_{(p,2)}(\ell_2^m)=\frac{1}{{2^{1+\frac{1}{p}}}},$ where $p>2.$\\
(iv) $n_{(p,2)}(\ell_2)=\frac{1}{{2^{1+\frac{1}{p}}}},$ where $p>2.$
\end{theorem}
\begin{proof}
(i) It follows from Theorem \ref{th-boundnpkx}	that 
\begin{equation}\label{eq-l201}
\frac{1}{2\sqrt{k}}=\frac{n(\ell_2^m)}{\sqrt{k}}\leq n_{(2,k)}(\ell_2^m).	
	\end{equation}
To prove the reverse inequality, we exhibit $\mathcal{T}\in L(\ell_2^m)$ such that $\|\mathcal{T}\|_2=1$ and $w_2(\mathcal{T}) =\frac{1}{2\sqrt{k}}.$ For this, we follow the idea of \cite[Th. 9]{DM}. Suppose $\{e_1,e_2,\ldots,e_m\}$ is the standard ordered basis of $\ell_2^m.$ For $1\leq i\leq k,$ define $T_i\in L(\ell_2^m)$ as follows.
\[T_ie_j=\begin{cases}
	&\theta, ~\text{if}~1\leq j<m,\\
	&\frac{1}{\sqrt{k}}e_i,~\text{if} ~j=m.
	\end{cases}\] 
Let $\mathcal{T}=(T_1,T_2,\ldots,T_k)\in L(\ell_2^m)^k.$ Then 
\begin{align*}
\|\mathcal{T}\|_2&=\sup\Big\{\big(\sum_{i=1}^k\|T_i(x_1,x_2,\ldots,x_m)\|^2\big)^{1/2}:\sum_{j=1}^m|x_j|^2=1\Big\}\\
&=\sup\{|x_m|:\sum_{j=1}^m|x_j|^2=1\}\\
&=1.
\end{align*}
Moreover, 
\begin{align*}
w_2(\mathcal{T})&=\sup\Big\{\big(\sum_{i=1}^k|\langle T_i(x_1,\ldots,x_m),(x_1,\ldots,x_k)\rangle|^2\big)^{1/2}:\sum_{j=1}^m|x_j|^2=1\Big\}	\\
&=\sup\Big\{\big(\sum_{i=1}^k|\frac{x_mx_i}{\sqrt{k}}|^2\big)^{1/2}:\sum_{j=1}^m|x_j|^2=1\Big\}	\\
&=\frac{1}{\sqrt{k}}\sup\Big\{\big(|x_m|^2\sum_{i=1}^k|x_i|^2\big)^{1/2}:\sum_{j=1}^m|x_j|^2=1\Big\}\\
&=\frac{1}{\sqrt{k}}\sup\Big\{\big(|x_m|^2\sum_{i=1}^k|x_i|^2\big)^{1/2}:\sum_{j=1}^k|x_j|^2+|x_m|^2=1\Big\}\\
&=\frac{1}{\sqrt{k}}\sup\Big\{\big(|x_m|^2(1-|x_m|^2)\big)^{1/2}:|x_m|\leq1\Big\}\\
&=\frac{1}{2\sqrt{k}}.
\end{align*}
Thus, 
\begin{equation}\label{eq-l202}
	n_{(2,k)}(\ell_2^m)\leq w_2(\mathcal{T})=\frac{1}{2\sqrt{k}}.
	\end{equation}
Combining (\ref{eq-l201}) and (\ref{eq-l202}), we get the conclusion of (i).\\

(ii) Note that, $\ell_2=[\ell_2^{k+1}\oplus\ell_2]_{\ell_2}.$ Therefore,
\[\frac{1}{2\sqrt{k}}=\frac{n(\ell_2)}{\sqrt{k}}\leq n_{(2,k)}(\ell_2)\leq n_{(2,k)}(\ell_2^{k+1})=\frac{1}{2\sqrt{k}},\]
where the first inequality follows from Theorem \ref{th-boundnpkx}, the second inequality follows from Theorem \ref{th-directsum} and the last equality follows from (i). This completes the proof of (ii).\\

(iii) Suppose $p>2.$ If we can prove $n_{(p,2)}(\ell_2^3)=\frac{1}{{2^{1+\frac{1}{p}}}},$ then for  $m>3,$ writing $\ell_2^m=[\ell_2^3\oplus \ell_2^{m-3}]_{\ell_2}$ and using Theorem \ref{th-boundnpkx} and Theorem \ref{th-directsum}, we get
\[\frac{1}{2^{1+\frac{1}{p}}}=\frac{n(\ell_2^m)}{2^{1/p}}\leq n_{(p,2)}(\ell_2^m)\leq n_{(p,2)}(\ell_2^3)=\frac{1}{2^{1+\frac{1}{p}}},\]
proving the result for $m>3.$ Thus, we only assume that $m=3.$ Clearly, 
\begin{equation}\label{eq-l203}
	\frac{1}{2^{1+\frac{1}{p}}}=\frac{n(\ell_2^3)}{2^{1/p}}\leq n_{(p,2)}(\ell_2^3).
\end{equation}  
To prove the reverse in equality, we use the same operators as in (i), i.e., we consider the operators $T_1,T_2\in L(\ell_2^3)$ defined as follows:
\[T_1(x,y,z)=\Big(\frac{z}{2^{1/p}},0,0\Big),~T_2(x,y,z)=\Big(0,\frac{z}{2^{1/p}},0\Big),~\forall~(x,y,z)\in \ell_2^3.\]
Suppose $\mathcal{T}=(T_1,T_2)\in L(\ell_2^3)^2.$ Then 
\begin{align*}
	\|\mathcal{T}\|_p&=\sup\Big\{\big(\|T_1(x,y,z)\|^p+\|T_2(x,y,z)\|^p\big)^{1/p}:|x|^2+|y|^2+|z|^2=1\Big\}\\
	&=\sup\{|z|:|x|^2+|y|^2+|z|^2=1\}\\
	&=1.
\end{align*}
On the other hand, 
\begin{align*}
&w_p(\mathcal{T})\\
&=\sup\Big\{\big(|\langle T_1(x,y,z),(x,y,z)\rangle|^p+|\langle T_2(x,y,z),(x,y,z)\rangle|^p\big)^{\frac{1}{p}}:|x|^2+|y|^2+|z|^2=1\Big\}	\\
&=\frac{1}{2^{\frac{1}{p}}}\sup\Big\{(|zx|^p+|zy|^p)^{\frac{1}{p}}:|x|^2+|y|^2+|z|^2=1\Big\}\\
&=\frac{1}{2^{1+\frac{1}{p}}},~(\text{since ~}p>2).
\end{align*}
Therefore, 
\begin{equation}\label{eq-l204}
	n_{(p,2)}(\ell_2^3)\leq w_p(\mathcal{T})=\frac{1}{2^{1+\frac{1}{p}}}.
	\end{equation}
Now, it follows from (\ref{eq-l203}) and (\ref{eq-l204}) that $n_{(p,2)}(\ell_2^3)=\frac{1}{2^{1+\frac{1}{p}}}.$ This completes the proof of (iii).\\

(iv) Proceeding similarly as (ii) and using (iii), we get the desired result.

\end{proof}

		

		\bibliographystyle{amsplain}

\begin{thebibliography}{99}
			
		\bibitem{B}	F. L. Bauer, \textit{On the field of values subordinate to a norm}, Numer. Math.,  \textbf{4 } (1962), 103-111.
		
			\bibitem{BD} F. F. Bonsall and J. Duncan, \textit{Numerical ranges of operators on normed spaces and of elements of normed algebras}, London Math. Soc. Lecture Note Series, 2 Cambridge Univ. Press, London-New York, (1971).
			
			\bibitem{BD2} F. F. Bonsall and J. Duncan, \textit{Numerical ranges II}, London Math. Soc. Lecture Note Series, No. 10, Cambridge Univ. Press, New York-London, (1973).
			
			\bibitem{CKS} J. Chmieli\'nski, D. Khurana and D. Sain, \textit{Approximate smoothness in normed linear spaces}, arXiv:2109.11884v1 [math.FA], (2021).
			
	\bibitem{DM} R. Drnov\v{s}ek and V. M\"uller, \textit{On joint numerical radius II}, Linear Multilinear Algebra, \textbf{62} (2014), no. 9, 1197-1204.
			
			\bibitem {GR} K. E. Gustafson and D. K. M. Rao, \textit{Numerical range. The field of values of linear operators and matrices}, Springer-Verlag, New York, (1997).
			
			\bibitem{H} P. R. Halmos, \textit{A Hilbert space problem book}, Van Nostrand, N.J.-Toronto, Ont.-London, (1967).
			
	
		\bibitem{KMP} V. Kadets, M. Mart\'in and R. Pay\'a, \textit{Recent progress and open questions on the numerical index of Banach spaces}, Rev. R. Acad. Cienc. Exactas Fís. Nat. Ser. A Mat., \textbf{100} (2006), no. 1-2, 155-182. 
			
			\bibitem{L} G. Lumer, \textit{Semi-inner-product spaces}, Trans. Amer. Math. Soc. \textbf{100} (1961), 29-43.
			
 \bibitem{MMQ} M. Mart\'in, J. Mer\'i and A. Quero, \textit{Numerical index and Daugavet property of operator ideals and tensor products}, Mediterr. J. Math., \textbf{18} (2021), no. 2, Paper No. 72, 15 pp.
		
		\bibitem{MO} M. Mart\'in and T. Oikhberg, \textit{An alternative Daugavet property}, J. Math. Anal. Appl., \textbf{294} (2004), no. 1, 158-180.
		
		\bibitem{MP} M. Mart\'in and R. Pay\'a, \textit{Numerical index of vector-valued function spaces}, Studia Math., \textbf{142} (2000), no. 3, 269-280.
		
  	\bibitem{MQ} J. Mer\'i and A. Quero, \textit{On the numerical index of absolute symmetric norms on the plane}, Linear Multilinear Algebra, https://doi.org/10.1080/03081087.2020.1762532 (2020).
		
		\bibitem{MSS} M. S. Moslehian, M. Sattari and K. Shebrawi, \textit{Extensions of Euclidean operator radius inequalities}, Math. Scand., \textbf{120} (2017), no. 1, 129-144.
			
		\bibitem{MT} V. M\"uller and Yu. Tomilov, \textit{Joint numerical ranges: recent advances and applications minicourse by V. M\"uller and Yu. Tomilov}, With assistance from Nikolitsa Chatzigiannakidou, Concr. Oper., \textbf{ 7} (2020), no. 1, 133-154.
			
			
			\bibitem{P} G. Popescu, \textit{Unitary invariants in multivariable operator theory},  Mem. Amer. Math. Soc.,  \textbf{200} (2009), no. 941, vi+91 pp.
		
		\bibitem{SMS} A. Sheikhhosseini, M. S. Moslehian and K. Shebrawi, \textit{Inequalities for generalized Euclidean operator radius via Young’s inequality}, J. Math. Anal. Appl., \textbf{445} (2017), no. 2, 1516-1529.
		
		\bibitem{SPBB} D. Sain, K. Paul, P. Bhunia, and S. Bag,  \textit{On the numerical index of polyhedral Banach spaces} Linear Algebra Appl., \textbf{577 } (2019), 121-133. 
			
		\end{thebibliography}

	\end{document}